\documentclass[11pt,a4paper,oneside]{amsart}
\usepackage{amsfonts, amsmath, amssymb, amsthm}
\usepackage[foot]{amsaddr}
\usepackage[colorlinks=true,citecolor=blue]{hyperref}
\usepackage[margin=1.4in]{geometry}
\usepackage{times}
\usepackage{txfonts}
\usepackage[T1]{fontenc}
\usepackage{bbm}
\usepackage{mathtools}

\usepackage{graphicx}
\usepackage{enumerate}
\usepackage{verbatim}
\usepackage{tikz}
\usepackage{algorithmic}
\usepackage[normalem]{ulem}

\usetikzlibrary{decorations,decorations.pathreplacing}

\begingroup
    \makeatletter
    \@for\theoremstyle:=definition,remark,plain\do{%
        \expandafter\g@addto@macro\csname th@\theoremstyle\endcsname{%
            \addtolength\thm@preskip\parskip
            }%
        }
\endgroup

\newtheorem{theorem}{Theorem}
\newtheorem*{theorem*}{Theorem}
\newtheorem{lemma}[theorem]{Lemma}
\newtheorem{prop}[theorem]{Proposition}
\newtheorem*{prop*}{Proposition}

\newtheorem{corollary}[theorem]{Corollary}
\newtheorem*{lemma*}{Lemma}

\theoremstyle{definition}

\newtheorem*{definition*}{Definition}

\newtheorem*{remark*}{Remark}

\newcommand{\cal}{\mathcal}
\newcommand{\defi}[1]{{{\em{#1}}}}

\DeclareMathOperator{\conv}{conv}

\newcommand{\subp}[1]{\subsubsection*{\bf #1}}

\begin{document} 

\title{The fractional Helly number for separable convexity spaces}  

\author{Andreas F. Holmsen \and Zuzana Pat\'akov\'a }

\date{\today}

\address{Andreas F. Holmsen, 
 Department of Mathematical Sciences, 
 KAIST, 
 Daejeon, South Korea.  \hfill \hfill  \linebreak \and \linebreak
Discrete Mathematics Group,  Institute for Basic Sciences (IBS), Daejeon, South Korea. }
\email{andreash@kaist.edu}

 \address{Zuzana Pat\'akov\'a, Department of Algebra, Faculty of Mathematics and Physics, Charles University, Prague, Czech Republic}
\email{patakova@karlin.mff.cuni.cz}

\begin{abstract} 
A \emph{convex lattice set} in $\mathbb{Z}^d$ is the intersection of a convex set in $\mathbb{R}^d$ with the integer lattice $\mathbb{Z}^d$. A classical theorem of Doignon states that the \emph{Helly number} of $d$-dimensional convex lattice sets equals $2^d$, exponentially larger than the Helly number $d+1$ of ordinary convex sets in $\mathbb{R}^d$. By contrast, a remarkable theorem of B{\'a}r{\'a}ny and Matou{\v s}ek states that the \emph{fractional Helly number} of convex lattice sets drops back down to $d+1$, matching the classical fractional Helly theorem of Katchalski and Liu.
In this paper we generalize the B{\'a}r{\'a}ny--Matou{\v s}ek theorem to abstract convexity spaces (in the sense of van de Vel) that satisfy a suitable separation axiom. Our main result implies the following: if a separable convexity space has Radon number at most $r$, then its fractional Helly number is at most $2^{r}$. This bound is nearly tight, as illustrated by the case of \emph{box convexity} in $\mathbb{R}^d$, whose Radon number is $\Theta(\log d)$ and fractional Helly number equals $d+1$.
\end{abstract}

\maketitle 
\section{Introduction}
\subsection{Background}

\subp{The fractional Helly theorem} 
We say that a family of sets is \emph{intersecting} if the intersection of all its members is nonempty. 
The well-known theorem of Helly \cite{helly} asserts that if $F$ is a finite family of at least $d+1$ convex sets in $\mathbb{R}^d$ in which every $(d+1)$-tuple is intersecting, then the entire family $F$ is intersecting. 

Among the numerous generalizations and variations of Helly's theorem, a particularly important one is the \emph{fractional Helly theorem}, due to Katchalski and Liu \cite{katch-liu}. Roughly speaking, it states that if the family contains \emph{many} intersecting $(d+1)$-tuples, then the family contains a \emph{large} intersecting subfamily.   
More precisely, the theorem asserts that if $F$ is a finite family of convex sets in $\mathbb{R}^d$ with at least $\alpha\cdot \binom{|F|}{d+1}$ intersecting $(d+1)$-tuples, then $F$ contains an intersecting subfamily of size at least $\beta \cdot |F|$, where $\beta$ is a positive constant depending only on $d$ and $\alpha$. In fact, the works of Eckhoff \cite{eckhoff} and Kalai \cite{kalai} imply that one may take $\beta = 1-(1-\alpha)^{1/(d+1)}$, which is known to be  optimal.

The significance of the fractional Helly theorem in modern combinatorial convexity was demonstrated by Alon and Kleitman in their celebrated proof of the $(p,q)$-theorem \cite{akpq}, and the theorem has been further studied in, e.g., \cite{akmm}. It turns out that in these applications, what matters most is the qualitative conclusion -- that a positive fraction of $F$ forms an intersecting subfamily -- rather than the precise value of $\beta$ in terms of $d$ and $\alpha$. 
See also \cite[Chapters 8--10]{mato} and \cite{bara-book} for further discussion and references. 

\subp{The fractional Helly property} Our goal in this paper is to investigate the fractional Helly theorem for set systems that are more general than the set system of all convex subsets of $\mathbb{R}^d$. In such a setting, it might not be clear what should play the role of the dimension, and we cannot hope to determine the exact dependency between $\alpha$ and $\beta$. Instead, we aim for \emph{existence results}, motivating the following definitions, first introduced in \cite{akmm}.

We say that a set system $\cal C$ satisfies the \defi{fractional Helly property} if there exist an integer $k$ and a function $\beta:(0,1) \to (0,1)$ for which the following holds: every finite family $F \subset {\cal C}$ that contains at least $\alpha\cdot\binom{|F|}{k}$ intersecting $k$-tuples, also contains an intersecting subfamily of size at least $\beta(\alpha)\cdot |F|$. In this situation, we also say that $\cal C$ satisfies the \defi{fractional Helly property for $k$-tuples}.  The smallest integer $k$ such that $\cal C$ satisfies the fractional Helly property for $k$-tuples (if it exists) is called the \defi{fractional Helly number} for $\cal C$.

Note that the fractional Helly property is trivially satisfied if $\cal C$ is finite, and so it is really only interesting to consider the fractional Helly property for infinite set systems. A natural example of such an infinite set system is the collection of all convex sets in $\mathbb{R}^d$, in which case the fractional Helly number equals $d+1$.

\subp{Convex lattice sets} 
Let us consider another, rather remarkable  example. 
A subset $S\subset \mathbb{Z}^d$ is a \emph{convex lattice set} if there exists a convex set 
$C\subset \mathbb{R}^d$ such that $S = C \cap \mathbb{Z}^d$. 
There is a Helly theorem for finite families of convex lattice sets, 
due to Doignon \cite{JPD}, asserting that if every $2^d$ members have a point in common, 
then the family is intersecting. 
The theorem is sharp, as $2^d$ cannot be replaced by any smaller value.

Regarding the fractional Helly property, it was observed in \cite[Section 8]{akmm} that the method of Katchalski and Liu (a collapsibility argument) can be used to show that the set system of all convex lattice sets in $\mathbb{Z}^d$ satisfies the fractional Helly property for $2^d$-tuples. However, the following result, due to B{\'a}r{\'a}ny and Matou{\v s}ek \cite{lattice}, shows that the fractional Helly number is considerably smaller.

\begin{theorem*}[B{\'a}r{\'a}ny--Matou{\v s}ek] \label{BM-lattice}
The fractional Helly number for the set system of all convex lattice sets in $\mathbb{Z}^d$ equals $d+1$.
\end{theorem*}

B{\'a}r{\'a}ny and Matou{\v s}ek's result also implies a $(p,q)$-theorem for convex lattice sets in $\mathbb{Z}^d$ for all integers $p\geq q \geq d+1$ \cite[Theorem 1.2]{lattice}. Therefore, although lattice convexity and standard convexity differ greatly in terms of their Helly numbers, for more ``global'' Helly-type phenomena, the key quantity remains $d+1$. 

Their proof of this result relies primarily on tools from extremal combinatorics, and in \cite[page 234]{lattice}, B{\'a}r{\'a}ny and Matou{\v s}ek make the following remark:

\begin{quote}
\emph{``Our proof of the fractional Helly theorem for convex lattice sets does not use much of the geometric properties of $\mathbb{Z}^d$ ... It would be interesting to clarify what axioms are sufficient ... say in the context of abstract convexity spaces."}
\end{quote}

The purpose of this paper is to address the remark of B{\'a}r{\'a}ny and Matou{\v s}ek, and thereby extend their theorem to the setting of abstract convexity spaces. 

\subsection{The main result}
To state our main result, we need to introduce several notions and parameters.  
\subp{Convexity spaces}
A convexity space $(X, \mathcal{C})$ consists of a non-empty set $X$ and a collection of subsets $\mathcal{C}\subset 2^X$ satisfying the following properties.
\begin{itemize}
\item[(C1)] $\emptyset$ and  $X$ are in $\mathcal{C}$.
    \item[(C2)] $\mathcal{C}$ is closed under intersections, that is, if $\mathcal{D}\subset \mathcal{C}$ is non-empty, then  $\bigcap_{C\in \mathcal{D}}C$ is in $\mathcal{C}$.
    \item[(C3)] $\mathcal{C}$ is closed under nested unions, that is, if $\mathcal{D}\subset \mathcal{C}$ is non-empty and totally ordered by inclusion, then $\bigcup_{C\in \mathcal{D}}C$ is in $\mathcal{C}$.
\end{itemize}
In a convexity space $(X, \mathcal{C})$ the members of $\mathcal{C}$ are called \emph{convex sets}. Given a subset $Y\subset X$ we define the \emph{convex hull} of $Y$, denoted by $\conv(Y)$, as the intersection of all convex sets that contain $Y$. For a comprehensive treatment of the theory of convexity spaces we refer the reader to \cite{vandevel}.

\subsubsection*{Remark}
Note that a set system that satisfies only properties (C1) and (C2) is often called a \emph{closure system}, and these properties alone suffice to define a closure (hull) operator. Property (C3) is equivalent to requiring that the closure operator is \emph{domain finite}, meaning that for any $A\subset X$ and $p \in \conv(A)$ there exists a finite subset $B\subset A$ such that $p\in \conv(B)$. (See \cite[Section 1.2]{vandevel}.)

The fractional Helly property is defined for \emph{finite} families of convex sets (in a typically infinite convexity space), and condition (C3) is in fact not essential for it. Consequently, our results extend in a straightforward way to general closure systems. However, to maintain a clear connection to the geometric origins, we choose to formulate and discuss our results in the setting of convexity spaces.

\subp{Radon and Helly numbers}
The complexity of a convexity space can be measured by various parameters. One of the crucial ones for us is the \emph{Radon number}, which is defined to be the smallest integer $n$ (if it exists) such that every subset $Y\subset X$ with $|Y|=n$ admits a partition $Y = A\cup B$ such that $\conv(A)\cap \conv(B)\neq \emptyset$. 

A closely related parameter is the \emph{Helly number}, the smallest integer $n$ (if it exists) such that every finite family $F\subset \mathcal{C}$ with $\bigcap_{S\in F}S = \emptyset$ contains a subfamily $G\subset F$ with $|G|\leq n$ and $\bigcap_{S\in G}S = \emptyset$. A fundamental theorem of Levi \cite{levi} asserts that for any convexity space with bounded Radon number, the Helly number exists and is strictly less than the Radon number.


\subp{Separability} Given a convexity space $(X, \mathcal{C})$, 
a \emph{halfspace} is a convex set $\gamma\in \mathcal{C}$ such that its complement is also convex, 
that is,  $\bar{\gamma} = X\setminus \gamma \in \mathcal{C}$. 
Note that $\emptyset$ and $X$ are both halfspaces. 
The convexity space $(X, \mathcal{C})$ is called \emph{separable} if for any 
$S\in \mathcal{C}$ and $x\in X\setminus S$ there exists a halfspace $\gamma\in \mathcal{C}$ 
such that $S\subset \gamma$ and $x\in \bar{\gamma}$. 
(In general, there exists a hierarchy of ``separation axioms'', introduced by Jamison in 1974 \cite{jamison}, 
and what we define here as separable is referred to as $S\hspace{-.2ex}_3$-separability in \cite[Chapter I.3]{vandevel}. See also \cite[Section~1]{chepoi24} for an in-depth historical overview of separability in convexity spaces.)

A useful fact is that a convexity space is separable if and only if the convex hull of any finite set can be expressed as an intersection of halfspaces. (See e.g. \cite[Proposition 1.7.3 and Corollary 3.9]{vandevel}.)

Standard convexity in $\mathbb{R}^d$ is a prototypical separable convexity space, 
and the collection of convex lattice sets in $\mathbb{Z}^d$ provides another example. 
By contrast, the convexity space formed by convex sets in $\mathbb{R}^d$ of diameter at most one is not separable: 
its only halfspaces are $\emptyset$ and $\mathbb{R}^d$. 
Our goal is to show that for any separable convexity space with bounded Radon number, 
the fractional Helly number can be bounded in terms of a natural complexity measure of its halfspace collection.

\subp{Dual VC-dimension} 
Let $\mathcal{F}$ be a set system on a ground set $X$. We denote by $A$ the incidence matrix of $\mathcal{F}$ on $X$, where the rows are indexed by elements of $\mathcal{F}$ and the columns by elements of $X$. The \emph{dual shatter function} of $\mathcal{F}$ is a function $\pi_\mathcal{F}^* : \mathbb{N} \to \mathbb{N}$, where $\pi_{\cal F}^*(m)$ is the maximum number of distinct columns in an $m$-row submatrix $A'$ of $A$ (the maximum is taken over all $m$-row submatrices). Note that $\pi_{\cal F}^*(m)$ counts the maximum number of nonempty regions of the Venn diagram of $m$ sets from $\cal F$. 

The \emph{dual VC-dimension} of $\cal F$ is the largest $k$ such that some $k$ members of $\mathcal{F}$ have a complete Venn diagram, that is, $\max\{ k : \pi_\mathcal{F}^*(k) = 2^k \}$.

\medskip
\noindent We are now in a position to state the main result of this paper:

\begin{theorem}\label{main}
Let $(X, {\cal C})$ be a separable convexity space with bounded Radon number,  and suppose the system of halfspaces has dual VC-dimension $d$. Then the fractional Helly number for $\cal C$ is at most $d+1$.  
\end{theorem}

The proof of Theorem \ref{main} will be given in the next section. For the rest of this section, we discuss some consequences of Theorem \ref{main} and how it relates to earlier results in the literature. 

\subsection{Consequences and relations to earlier work}

\subp{Matou{\v s}ek's theorem} A connection between the dual VC-dimension and the fractional Helly property was previously established by Matou{\v s}ek \cite{mato-vc}:

\begin{theorem*}[Matou{\v s}ek] \label{thm:mato}
Let $\cal F$ be a set system with dual VC-dimension $d$. Then $\cal F$ has fractional Helly number at most $d+1$.
\end{theorem*}

It is instructive to compare this to our main result (Theorem \ref{main}). While Matou{\v s}ek's theorem gives a bound on the fractional Helly number for \emph{any} set system $\cal F$ with bounded dual VC-dimension, our result, on the other hand, should be viewed as a statement about \emph{symmetric} set systems $\cal F$, satisfying $\gamma\in \cal F \Rightarrow \bar{\gamma}\in \cal F$. 

Both theorems yield the same bound on the fractional Helly number in terms of the dual VC-dimension. However, the main strength of our result is that the fractional Helly property applies not only to $\cal F$, but to the entire set system $\cal F^\cap  = \big\{  \bigcap_{A\in \cal G}A : \cal G \subset \cal F   \big\}$. For instance, the set system of convex sets in $\mathbb{R}^d$ does not have bounded (dual) VC-dimension, and so Matou{\v s}ek's theorem does not imply the fractional Helly property for convex sets in $\mathbb{R}^d$. On the other hand, the set system of convex sets in $\mathbb R^d$ is separable \cite{hemispaces} and it is straightforward to check that the dual VC-dimension of the system of halfspaces in $\mathbb{R}^d$ equals $d$ (see e.g., \cite[Section 5.1]{mato-gd}). Theorem \ref{main} implies that convex sets in $\mathbb{R}^d$ have fractional Helly number (at most) $d+1$.

\subp{The B{\'a}r{\'a}ny--Matou{\v s}ek theorem} It follows readily that  B{\'a}r{\'a}ny and Matou{\v s}ek's theorem on the fractional Helly number for convex lattice sets is a special case of Theorem \ref{main} as well. The set system of convex lattice sets in $\mathbb{Z}^d$ is separable, with halfspaces of the form $\mathbb{Z}^d\cap \gamma$, where $\gamma$ is a halfspace in $\mathbb R^d$. The dual VC-dimension of halfspaces in $\mathbb Z^d$ is $d$ (by the same argument as for $\mathbb R^d$). Furthermore, a result due to Onn \cite{onn} shows that the Radon number for convex lattice sets in $\mathbb{Z}^d$ is at most $d(2^d-1)+3$, and so Theorem \ref{main} applies. Note that the precise Radon number for convex lattice sets in $\mathbb{Z}^d$ is not known for $d\geq 3$; for $d = 2$ the Radon number equals 6, but for $d\geq 3$ the best known lower bound, due to Sierksma  \cite{sierksma}, is $5\cdot 2^{d-2}+1$.

\subsubsection*{Remark} One might ask whether, in a separable convexity space with bounded Radon number in which the system of halfspaces has dual VC-dimension $d$, the \emph{Helly number} is at most $d+1$. However, this is not true, as the set system of convex lattice sets in $\mathbb Z^d$ satisfies the assumptions but has Helly number $2^d$ \cite{JPD}.

\subp{Fractional Helly and Radon numbers}  Holmsen and Lee \cite{holmsen-lee} showed that for a general convexity space (not necessarily separable), it is possible to bound the fractional Helly number in terms of its Radon number:

\begin{theorem*}[Holmsen--Lee] Let $(X, \cal C)$ be a convexity space with Radon number $r\geq 3$. Then the fractional Helly number for $\cal C$ is at most $r^{r^{\log r}}$. \end{theorem*}

As we will demonstrate below, by adding the separability assumption we can use Theorem \ref{main} to reduce the bound on the fractional Helly number to exponential in $r$. 

The complexity of the system of halfspaces of a separable convexity space was previously investigated by Moran and Yehudayoff \cite{moran}, who used the  {\em VC-dimension} (primal, not the dual) to establish a weak $\varepsilon$-net theorem for convex sets in a separable convexity space with bounded Radon number. 

The (primal) {\em VC-dimension} of a set system $\cal F$ is defined as  $\max\{ k : \pi_{\cal F}(k) = 2^k \}$, where $\pi_{\cal F}(m)$ is defined as the maximum number of distinct rows in any $m$-column submatrix of the incidence matrix $A$ of $\mathcal{F}$ on $X$. (Recall that rows of $A$ are indexed by members of $\mathcal F$ and columns by elements of $X$, and compare with the definition of the dual shatter function.)

There is a simple relation between the  primal and dual VC-dimensions, first noted by Assouad \cite{assouad}. (For a proof in English, see \cite[Lemma 10.3.4]{mato}).

\begin{lemma*}[Assouad]\label{eq:vc-dim}
For any set system $\mathcal F$ we have
\[\text{\emph{dual VC-dim}}(\mathcal{F}) < 2^{\text{\emph{VC-dim}} (\mathcal{F}) + 1}.\]
\end{lemma*}

As for separable convexity spaces, Moran and Yehudayoff \cite[Lemma 1.5]{moran} observed the following:

\begin{lemma*}[Moran--Yehudayoff]\label{MYobs}
    Let $(X, \cal C)$ be a separable convexity space with Radon number $r\geq 3$. Then the system of halfspaces has VC-dimension at most $r-1$. 
\end{lemma*}

By combining the Lemmas of Assouad and Moran--Yehudayoff with Theorem \ref{main}, we obtain the following:

\begin{corollary}\label{c:frac}
 Let $r \geq 3$ and $(X, \mathcal C)$ be a separable convexity space with Radon number $r$. Then the fractional Helly number for $\mathcal C$ is at most $2^r$.
\end{corollary}

We are not aware of any non-separable convexity space with a fractional Helly number that is asymptotically greater than $2^r$.

To see that the bound in Corollary \ref{c:frac} is (asymptotically) tight, consider the example of \emph{box convexity} on $\mathbb{R}^d$. This is the convexity space whose convex sets are  axis-parallel boxes in $\mathbb{R}^d$, that is, Cartesian products of \textit{intervals} (including degenerate ones: points, rays, lines) parallel to the coordinate axes. 

This space is separable; the nontrivial halfspaces are 
of the form $\{ x \in \mathbb{R}^d: x_i\leq c \}$ and $\{ x\in \mathbb{R}^d: x_i\geq c \}$, together with  their complements. 
The Radon number of this space equals 
\[\min \left\{ r: \binom{r}{\lfloor r/2 \rfloor} > 2d \right\},\]
which grows asymptotically as $\Theta(\log_2d)$. (See \cite[Section 3]{eckhoff-box} for more details.)
The fractional Helly number for box convexity in $\mathbb{R}^d$ is obviously bounded above by $d+1$, by the standard fractional Helly theorem, and it is easily seen that this bound cannot be reduced: Take, for instance, $n$ axis-parallel hyperplanes split into $d$ groups of size $n/d$, where the members of the $i$th group are orthogonal to the $i$th coordinate axis. This is a family of $n$ boxes which contains $(n/d)^d \geq \frac{1}{e^d}\binom{n}{d}$ intersecting $d$-tuples, but where the largest intersecting subfamily has size $d$.

\subp{Necessity of bounded Radon number}
Finally, we discuss the necessity of bounding the Radon number in Theorem \ref{main} through some specific separable convexity spaces suggested by B{\'a}r{\'a}ny and Kalai  \cite[Conjecture 2.9 and Problem 2.10]{BK-bull}. 

For positive integers $k$ and $d$, let $\mathcal{B}_k^d$ denote the family of solution sets in $\mathbb{R}^d$ of systems of polynomial inequalities (strict and nonstrict) of degree at most $k$. Formally, $\mathcal{B}_k^d$ is a closure system; closing under nested unions yields a convexity space. This convexity space is separable: The halfspaces are the solution sets in $\mathbb{R}^d$ of a {\em single} polynomial inequality. Note that in the case $k=1$, the halfspaces are given by linear inequalities, which gives us the standard convexity on $\mathbb{R}^d$.

B{\'a}r{\'a}ny and Kalai \cite[Conjecture 2.9]{BK-bull} conjectured that for all integers $d$ and $k$, the space $\mathcal{B}_k^d$ has the fractional Helly property. We show that this fails already for $d=1, k=2$.

First we show that the convexity space $\mathcal{B}_2^1$ is ``universal'' in the following sense:

\begin{prop} \label{BK-counter}
Let $\mathcal{F} = \{S_1, S_2, \dots, S_m \}$ be a finite set system on a finite ground set $X$. There exists a family of sets  $\{C_1, C_2, \dots, C_m\} \subset \mathcal{B}_2^1$ with the same intersection pattern as $\mathcal{F}$: that is, 
for every $I\subset [m]$,  $\bigcap_{i\in I}S_i \neq \emptyset$ if and only if $\bigcap_{i\in I} C_i\neq \emptyset$.
\end{prop}

Assuming this proposition, unboundedness of the fractional Helly number of $\mathcal{B}_2^1$ can be shown as follows: for any given integer $t\geq 2$ there exist finite set systems of arbitrarily large size in which every $t$-tuple is intersecting, but every $(t+1)$-tuple has empty intersection -- the strongest possible failure of the fractional Helly property. In particular, this disproves the conjecture of B\'ar\'any and Kalai \cite[Conjecture 2.9]{BK-bull}. 
Moreover, the same construction shows that the Helly and Radon numbers of ${\cal B}_2^1$ are unbounded as well.

\begin{proof}[Proof of Proposition \ref{BK-counter}] Observe that for any real numbers $a_1 < a_2 < \cdots < a_{n+1}$, the system of linear and quadratic inequalities in $\mathbb{R}$ 
\[\{a_1 \leq x \leq a_{n+1}\} \cup   \big\{0 \leq  (x-a_i) (x-a_{i+1})\big\}_{i=1}^n,\]
has the points $a_1, a_2, \dots, a_{n+1} \in \mathbb{R}$ as its set of solutions, and therefore any finite subset of the real line is a member of $\mathcal{B}_2^1$. 

To complete the proof, for the given set system $\cal F   = \{S_1, S_2, \dots, S_m\}$, we index the nonempty regions of the Venn diagram $V(\cal F)$ of $\cal F$ by distinct real numbers $b_j$, and for every $i \in [m]$, we let $C_i \in \mathcal{B}_2^1$ consist of all points $ b_j \in \mathbb{R}$ where $b_j$ is assigned to a region of $V(\cal F)$ contained in $S_i$. 
 
An intersection $\bigcap_{i\in I}C_i$ is nonempty if and only if it contains some point $b_j$. By construction, this is equivalent to a nonempty region of $V(\cal F)$, indexed by $b_j$, that is contained in $S_i$ for every $i\in I$, which in turn is equivalent to $\bigcap_{i\in I}S_i \neq \emptyset$. 
\end{proof}

It follows by a standard lifting argument that the system of halfspaces of $\mathcal{B}_k^d$ has dual VC-dimension bounded by a constant depending on $k$ and $d$. Indeed, by linearization, a single polynomial inequality (of degree at most $k$) becomes a linear inequality in $\mathbb R^{\binom{d+k}{d}}$, and as the dual VC-dimension of standard Euclidean (linear) halfspaces is finite (depending only on the dimension), the claim follows. 

As we have seen, both Radon and fractional Helly numbers of ${\cal B}_2^1$ are unbounded, even though the system of halfspaces has finite dual VC-dimension.  The space $\mathcal{B}_2^1$ therefore demonstrates the necessity of a bounded Radon number in Theorem \ref{main}.

Furthermore, B{\'a}r{\'a}ny and Kalai asked whether one can bound the fractional Helly number of a separable convexity space whose system of halfspaces has bounded VC-dimension \cite[Problem 2.10]{BK-bull}. (They refer to this as an MY-space.) In view of Assouad's lemma and Proposition \ref{BK-counter}, this question has no affirmative answer without a bounded Radon number. Indeed, it follows from Theorem \ref{main} and Assouad's lemma that this additional assumption is sufficient: If the Radon number of a separable convexity space $(X,\mathcal{C})$ is bounded and the system of halfspaces has VC-dimension at most $d$, then the dual VC-dimension is strictly less than $2^{d+1}$, and consequently the fractional Helly number for $\mathcal{C}$ is at most $2^{d+1}$.

There is also a second part to B{\'a}r{\'a}ny and Kalai's question \cite[Problem 2.10]{BK-bull}, asking whether the fractional Helly number of a separable convexity space is \textit{equal} to the VC-dimension of its system of halfspaces. This is not true even under the assumption that the Radon number is bounded. To see this consider Box convexity. It was shown by Gey \cite{gey} that the VC-dimension of axis-parallel halfspaces in $\mathbb{R}^d$ is of order $\log d$, and as noted earlier, the fractional Helly number equals $d+1$.


\section{Proof of the main result}
Let $F_1, \dots, F_k$ be finite families of subsets of some ground set $X$. We say that the families $F_1, \dots, F_k$ satisfy the \emph{colorful Helly hypothesis} if $\bigcap_{i=1}^{k}C_i \neq \emptyset$ for every choice $C_1\in F_1, \dots, C_{k}\in F_{k}$. The classical colorful Helly theorem of B{\'a}r{\'a}ny \cite{col-hell} asserts that if $k = d+1$ finite families of convex sets in $\mathbb{R}^d$ satisfy the colorful Helly hypothesis, then one of the families must be intersecting. 

The colorful Helly theorem was generalized to convexity spaces (not necessarily separable) by Holmsen and Lee \cite{holmsen-lee}. They define the \emph{colorful Helly number} of $(X, \mathcal{C})$ as the smallest $k$ such that whenever $k$ finite families $F_1$, $F_2$, $\dots$, $F_k \subset \mathcal{C}$ satisfy the colorful Helly hypothesis, at least one of the $F_i$ is intersecting. Specifically, Holmsen and Lee show that if $(X, \mathcal{C})$ has Radon number at most $r$, then its \emph{colorful Helly number} is at most $k = k(r)\leq r^{r^{\log r}}$. Note that this is the same
bound as the one on the fractional Helly number in terms of its Radon number \cite{holmsen-lee}, which is not a coincidence: a bounded colorful Helly number implies a bounded fractional Helly number \cite{holmsen}.

The key step in the proof of Theorem \ref{main} is a ``weak'' colorful Helly theorem. The difference from the usual colorful Helly theorem is that instead of concluding that one of the families is intersecting, we only conclude that one of the families contains a sufficiently large intersecting \emph{subfamily}.


\begin{prop}\label{weak-colorful}
For any integers $r\geq 3$, $d\geq 1$ and $m\geq 2$ there exists a positive integer $p = p(r,d,m)$ with the following property:
if $(X, \mathcal{C})$ is a separable convexity space with Radon number $r$, the system of halfspaces has dual VC-dimension $d$, and $F_1, \dots, F_{d+1} \subset \mathcal{C}$ are families with $|F_i| = p$ satisfying the colorful Helly hypothesis, then one of the families $F_i$ contains $m$ members with nonempty intersection.  
\end{prop}

The proof of Proposition \ref{weak-colorful} closely follows the argument of B{\'a}r{\'a}ny and Matou{\v s}ek \cite[Proposition 3.1]{lattice}, who established Proposition \ref{weak-colorful} for the case of convex lattice sets. This is the main technical step of this paper and is given in Section \ref{section:weak proof}. 

With Proposition \ref{weak-colorful} at hand, Theorem \ref{main} now follows from a standard counting argument using the  Erd{\H o}s--Simonovits supersaturation theorem  \cite{ErdSim}. The theorem is used frequently in discrete geometry and a simple proof can be found in \cite[Theorem 9.2.2]{mato}. For completeness, we spell out the proof of Theorem \ref{main} below.

For integers $k, s\geq 2$, let $K^k(s)$ denote the complete $k$-uniform $k$-partite hypergraph where each vertex part has size $s$.

\begin{theorem*}[Erd{\H o}s--Simonovits] \label{erdos-simonovits}
For all integers $k, s\geq 2$ and a real number $\varepsilon\in (0,1]$ there exists a $\delta>0$ with the following property. Let $H$ be a $k$-uniform hypergraph on $n$ vertices and at least  $\varepsilon \binom{n}{k}$ edges. Then $H$ contains at least $\lfloor\delta n^{ks}\rfloor$ copies of $K^k(s)$.
\end{theorem*}
Let us remark on an important special case: \emph{For any fixed $\varepsilon>0$, if a $k$-uniform hypergraph $H$ contains an $\varepsilon$-fraction of all possible edges and $n$ is sufficiently large in terms of $k$ and $\varepsilon$, then $H$ contains $K^k(s)$ as a subhypergraph (not necessarily induced).}

\begin{proof}[Proof of Theorem \ref{main} (assuming Proposition \ref{weak-colorful})]
Consider a separable convexity space $(X, \mathcal{C})$ satisfying the hypothesis of Theorem \ref{main}. Let $r$ be the Radon number and let $d$ be the dual VC-dimension of the system of halfspaces.  Let $m$ be the bound on the fractional Helly number given by the Holmsen--Lee theorem, and let $p = p(r,d,m)$ be the integer from Proposition \ref{weak-colorful}. We will show that if $F\subset \mathcal{C}$ is a finite family in which a positive fraction of the $(d+1)$-tuples are intersecting, then a positive fraction of the $m$-tuples of $F$ are intersecting.

Now suppose we are given a family $F\subset \mathcal{C}$, with $|F|=n$, which contains at least $\alpha\binom{n}{d+1}$ intersecting $(d+1)$-tuples. Let $H$ be a $(d+1)$-uniform hypergraph on $n$ vertices which are in one-to-one correspondence with the members of $F$, and let a $(d+1)$-tuple of vertices of $H$ form an edge whenever the corresponding $(d+1)$-tuple of $F$ is intersecting. Thus $H$ has at least $\alpha\binom{n}{d+1}$ edges, and so by the Erd{\H o}s--Simonovits theorem, $H$ contains at least $\delta n^{p(d+1)}$ distinct copies of $K^{d+1}(p)$ for some $\delta>0$ (which depends only on $\alpha$, $d$, and $p$). 
Moreover, a fixed $p(d+1)$-tuple of vertices can contain at most a constant number (depending only on $p$ and $d$) of distinct copies of $K^{d+1}(p)$, and so $H$ contains at least $\delta'n^{p(d+1)}$ copies of $K^{d+1}(p)$ each supported on a distinct $p(d+1)$-tuple of vertices, for some $\delta'>0$ (which still only depends on $\alpha$, $d$, and $p$). By Proposition \ref{weak-colorful}, for each such copy we find an $m$-tuple of $F$ with nonempty intersection. Each such intersecting $m$-tuple will be counted in this way by at most $\binom{n-m}{p(d+1)-m}$ distinct copies of $K^{d+1}(p)$, which means (as $n\to \infty$ with $m, p, d$ fixed) there are at least
\[\textstyle{\frac{\delta'n^{p(d+1)}}{\binom{n-m}{p(d+1)-m}} \geq cn^{m}\geq \alpha'\binom{n}{m}}\]
distinct $m$-tuples of $F$ with nonempty intersection, where $\alpha'>0$ depends only on $\alpha$, $d$, and $p$. We may now apply the Holmsen--Lee theorem to find $\beta n$ members of $F$ with nonempty intersection, where $\beta = \beta(\alpha', r) > 0$.
\end{proof}

\subsubsection*{Remark} 
Formally, we can define the \emph{weak colorful Helly number} of a convexity space $(X,\mathcal{C})$ to be the smallest integer $k\geq 2$ such that for every $m\geq 2$ there exists a $p = p(m)$ with the following property: if $F_1$, $\dots$, $F_{k} \subset \mathcal{C}$, with $|F_i|=p$, satisfy the colorful Helly hypothesis, then one of the families $F_i$ contains $m$ members with nonempty intersection. In this terminology, Proposition \ref{weak-colorful} asserts that the weak colorful Helly number of a separable convexity space with bounded Radon number is at most $d+1$, where $d$
is the dual VC-dimension of its system of halfspaces. 

Clearly the weak colorful Helly number is bounded above by the colorful Helly number, and it follows from \cite{holmsen} that the fractional Helly number of a convexity space is bounded above by its colorful Helly number. This is detailed in \cite{holmsen-lee} together with a bound on the colorful Helly number in terms of the Radon number.  

The preceding argument (showing that Theorem \ref{main} follows from Proposition \ref{weak-colorful}) actually tells us that the fractional Helly number is bounded above by the weak colorful Helly number, provided that the fractional Helly number is bounded. Even though the assumption of a bounded fractional Helly number is critical in the preceding argument, we are not aware of any convexity space in which the weak colorful Helly number is bounded but the fractional Helly number is unbounded.

\section{Auxiliary results} \label{section:weak proof}

The proof of Proposition \ref{weak-colorful} is the main technical step of this paper. 
It relies on two auxiliary results about finite point configurations in a separable convexity space. 
These results (Lemmas \ref{large separated pairs} and \ref{for any f} below) were originally proved by B{\'a}r{\'a}ny and Matou{\v s}ek in the setting of convex lattice sets in $\mathbb{Z}^d$. 
Surprisingly, in our generalizations of these lemmas, the dual VC-dimension of the system of halfspaces does not play any role, and we require only that the Radon number is bounded. 
The bound on the dual VC-dimension comes into play only in the very last step of the proof of Proposition \ref{weak-colorful}.

Let $(X, \mathcal{C})$ be a separable convexity space. Throughout this section we frequently deal with finite multisets of points from $X$. It will be convenient to think of a multiset as a function $f: E \to X$, where $E$ is a nonempty finite set. 
For a subset $Z\subset X$, we write $|Z|_f = \sum_{x\in Z}|f^{-1}(x)|$. In other words,  $|Z|_{f}$ counts, with multiplicity, the number of points of $f(E)$ that are contained in $Z$.

The following lemma extends \cite[Corollary 2.2]{lattice} to the setting of separable convexity spaces. 

\begin{lemma} \label{large separated pairs}
Let $(X, \mathcal{C})$ be a separable convexity space with Radon number at most $r$, let $E$ be a nonempty finite set, and let $m \geq 2$ be an integer. 
For every finite family of functions 
$\{f_i\}_{i=1}^m$, where $f_i:E\to X$
and $\bigcap_{i=1}^m \conv(f_i(E)) = \emptyset$, there exist 
\begin{itemize}
    \item[-] a subset $E_0\subset E$, with $|E_0|\geq \frac{1}{r-1}|E|$,
    \item[-] two functions $f_a$ and $f_b$ from the family, and
    \item[-] a halfspace $\gamma\in \mathcal{C}$,
\end{itemize}
such that  $f_a(E_0)\subset \gamma$ and $f_b(E_0)\subset \bar{\gamma}$.  
\end{lemma}

\begin{proof}
Let $S = \bigcup_{i=1}^mf_i(E)$ and set 
$N = \sum_{i=1}^m |S|_{f_i} = m|E|$. 
Define the family
\[
F=\left\{\: {\conv(Z) \; : \; Z\subseteq S} ;\:  \textstyle{\sum_{i=1}^m |Z|_{f_i} > \tfrac{r-2}{r-1} N} \right\}.
\]
Note that $F$ is finite since $S$ is finite, and nonempty since $\conv(S) \in F$. 
Moreover, every $k\leq r-1$ members of $F$ have a point in common: 
if $\conv(Z_1)$, $\dots$, $\conv(Z_k) \in F$, then $\sum_{i=1}^m |S\setminus Z_l|_{f_i}  < \frac{1}{r-1} N$ 
for every $1\leq l \leq k$, 
and so $\sum_{l=1}^k \sum_{i=1}^m|S\setminus Z_l|_{f_i} < \frac{k}{r-1} N\leq N$. If $\bigcap_{l=1}^k Z_l$ were empty, then $\bigcup_{l=1}^k (S\setminus Z_l) = S$. Since $|\cdot|_{f_j}$ is a counting measure, the covering $S = \bigcup_{l=1}^k (S\setminus Z_l)$ gives $|S|_{f_j} \leq \sum_{\ell=1}^k |S\setminus Z_\ell|_{f_j}$ for each $j$, and summing over all $j$ yields $N \leq \sum_{l=1}^k \sum_{i=1}^m|S\setminus Z_l|_{f_i} < N$, a contradiction. Hence $\bigcap_{l=1}^k Z_l\neq\emptyset$, and since $Z_l\subseteq\mathrm{conv}(Z_l)$ for each $l$, any point in $\bigcap_{l=1}^k Z_l$ lies in $\bigcap_{l=1}^k\mathrm{conv}(Z_l)$.

Thus $F$ is a nonempty finite family of convex sets in $\mathcal{C}$ in which every $k\leq r-1$ members have nonempty intersection. We have assumed that the Radon number of $(X, \mathcal{C})$ is at most $r$, and Levi's theorem \cite{levi} implies that the Helly number of $(X, \mathcal{C})$ is at most $r-1$. It follows that there is a point $x_0$ in common to all members of $F$.

Since $\bigcap_{i=1}^m \text{conv}(f_i(E)) = \emptyset$, 
there is an $f_a$ such that $x_0 \not\in \conv(f_a(E))$. 
By separability there is a halfspace $\gamma$ in $\cal C$ such that 
\[ f_a(E) \subseteq \conv(f_a(E)) \subseteq \gamma \: \text{ and } \: x_0\in \bar{\gamma}.\] 
Moreover, as $x_0$ belongs to every set in $F$ yet is not contained in the halfspace $\gamma$,
we have 
$\conv(\gamma \cap S) \notin F$, hence $\textstyle{ \sum_{j=1}^m |\gamma|_{f_j} \leq \frac{r-2}{r-1}N }$, which implies $\sum_{j=1}^m|\bar{\gamma}|_{f_j}\geq \frac{N}{r-1}$.

Note also that $|\bar{\gamma}|_{f_a}=0$ since $f_a(E) \subseteq \gamma$.
By the pigeonhole principle there is a $b\in [m]\setminus\{a\}$ such that 
\[\textstyle{|\bar{\gamma}|_{f_b} \geq \frac{N}{(m-1)(r-1)} = \frac{m|E|}{(m-1)(r-1)} \geq \frac{1}{r-1}|E|}.\] 
To finish the proof, we set $E_0= \{e\in E \: : \: f_b(e)\in \bar{\gamma}\}$.
\end{proof}

Let $H = (V, E)$ be a $k$-uniform hypergraph. 
For a vertex $v\in V$ and a subset $\tilde{E}\subset E$, we define the subset $\tilde{E}_v\subset \tilde{E}$ as 
\[\tilde{E}_v = \{e\in \tilde{E} \: : \: v\in e\}.\]

Fix a separable convexity space $(X, \mathcal{C})$, a $k$-uniform hypergraph $H = (V,E)$, and a function $f:E\to X$. For a subhypergraph $\tilde{H} = (\tilde{V}, \tilde{E})$ of $H$ and a subset of vertices $A\subset \tilde{V}$, we say that \emph{$A$ is separable with respect to $\tilde{H}$} if there exist 
\begin{itemize}
    \item[-] a pair of distinct vertices $u, v \in A$, and
    \item[-] a halfspace $\gamma\in \mathcal{C}$
\end{itemize}
with the property that \[f(\tilde{E}_u) \subset \gamma \: \text{ and } \: f(\tilde{E}_v) \subset \bar{\gamma}.\]
(Here $\tilde{E}_u$ and $\tilde{E}_v$ denote the edges of $\tilde{H}$ that contain $u$ and $v$, respectively.)

Note that this separability property is hereditary in the following sense: if $A$ is separable with respect to $\tilde{H}$ and $H' = (V', E')$ is a subhypergraph of $\tilde{H}$ such that $A\subset V'$, then $A$ is also separable with respect to $H'$ (using the same witnesses $u, v$, and $\gamma$ that witnessed the separability of $A$ with respect to $\tilde{H}$).

The following is an extension of \cite[Lemma 3.2]{lattice}.

\begin{lemma} \label{for any f}
Given integers $n\geq m\geq 2$, $k\geq 2$ and $r\geq 3$, 
there exists an integer $t=t(n,m,k,r)$ with the following property.
Let $(X, \mathcal{C})$ be a separable convexity space with Radon number at most $r$, and let $H = K^k(t)$ with vertex parts $V_1, \dots, V_k$ and edge set $E$. For every function $f:E\to X$, one of the following holds:
\begin{enumerate}
\item \label{intersecting case} There is an $m$-element set $A\subset V_{k}$  such that 
    \[\textstyle{\bigcap_{v\in A}} \conv(f(E_v)) \neq \emptyset.\]

\item \label{separation case} There are $n$-element sets $W_1\subset V_1, \dots, W_k\subset V_k$ such that, for the induced subhypergraph $\tilde{H} = H[\bigcup_{i=1}^k W_i]$, every $m$-element subset of $W_k$ is separable with respect to $\tilde{H}$.
\end{enumerate}
\end{lemma}

(Note the asymmetry in Lemma \ref{for any f} where the $k$th vertex part plays a special role.) 


\begin{proof}
Fix an arbitrary $n$-element set $W_k\subset V_k$, 
and fix an $m$-element subset $A\subset W_k$. 
Suppose that $A$ does not fall into case \emph{(\ref{intersecting case})}, that is,
\[\textstyle{\bigcap_{v\in A}}\conv(f(E_v)) = \emptyset.\]

We claim that for a given integer $s$, if $t$ is sufficiently large, then there exist
$s$-element subsets $U_1\subset V_1, \dots, U_{k-1}\subset V_{k-1}$ such that 
for the induced subhypergraph $H' = H[U_1 \cup \cdots \cup U_{k-1}\cup W_k]$, 
$A$ is separable with respect to $H'$.

Note that this claim will prove the lemma, 
because after applying the claim at most $\binom{n}{m}$ times and passing to smaller $U_i$'s at each step, 
we can make every $m$-element subset of $W_k$ separable with respect to the resulting induced subhypergraph 
$H[\bigcup_{i=1}^k W_i]$. By starting with $t$ sufficiently large, 
we can guarantee that $|W_1| = \cdots = |W_{k-1}| = n$.

To prove the claim, we apply Lemma \ref{large separated pairs}.
Let $\tilde{E}$ denote the edge set of the $(k-1)$-uniform hypergraph $K^{k-1}(t)$ 
on the vertex classes $V_1, \dots, V_{k-1}$. 
For every $w\in A$, define a function $f_w:\tilde{E} \to X$ by setting 
$f_w(\sigma) = f(\sigma\cup \{w\})$. Note that $f_w(\tilde{E}) = f(E_w)$. 
Since $\bigcap_{w\in A}\conv(f_w(\tilde{E})) = \bigcap_{w\in A}\conv(f(E_w)) = \emptyset$, the hypothesis of Lemma \ref{large separated pairs} is satisfied; by that lemma, there exist 
\begin{itemize}
    \item[-] a subset $\tilde{E}_0\subset \tilde{E}$, with $|\tilde{E}_0|\geq \frac{1}{r-1}|\tilde{E}|$,
    \item[-] a pair of distinct vertices $u,v\in A$, and
    \item[-] a halfspace $\gamma\in \mathcal{C}$,
\end{itemize}
such that $f_u(\tilde{E}_0)\subset \gamma$  and $f_v(\tilde{E}_0)\subset \bar{\gamma}$.

To complete the proof of the claim, we apply the Erd{\H o}s--Simonovits theorem. Since $\tilde{E}_0$ contains at least a $\frac{1}{r-1}$-fraction of the edges of $K^{k-1}(t)$ 
and hence at least an $\varepsilon$-fraction of all possible edges of a $(k-1)$-uniform hypergraph on $(k-1)t$ vertices, where $\varepsilon$ depends only on $k$ and $r$, it follows that if we choose $t$ sufficiently large, then $\tilde{E}_0$ contains the edge set of a copy of $K^{k-1}(s)$. Setting 
$U_1, \dots, U_{k-1}$ to be the vertex parts of this $K^{k-1}(s)$ 
gives $f(E'_u) \subset f_u(\tilde{E}_0) \subset \gamma$ and $f(E'_v) \subset f_v(\tilde{E}_0) \subset \bar{\gamma}$, witnessing the separability of $A$ with respect to $H'$. This
completes the proof of the claim.
\end{proof}

\section{Proof of Proposition \ref{weak-colorful}}

Recall that $F_1, \dots, F_{d+1}$ are families of convex sets in a separable convexity space $(X, \mathcal{C})$ with Radon number at most $r$ and where the system of halfspaces has dual VC-dimension at most $d$. 
For every $i\in [d+1]$ we have $|F_i|=p$, and for every choice $C_1\in F_1, \dots, C_{d+1}\in F_{d+1}$ we have $C_1\cap \cdots\cap C_{d+1}\neq \emptyset$. 
Assume for contradiction that for every $i\in [d+1]$, each $m$-tuple of $F_i$ has empty intersection. 

Let $H = K^{d+1}(p)$ with vertex parts $V_1, \dots, V_{d+1}$ and edge set $E$. For every vertex $v \in V_i$ we assign a convex set $C(v)\in F_i$, so that for every $i\in [d+1]$ we have a bijection between the vertices in $V_i$ and the members of $F_i$.
Define a function $f:E \to X$ by choosing an arbitrary point
\[\textstyle{ f(e) \in  \bigcap_{v\in e} C(v) },\]
which is possible by hypothesis. 

For sufficiently large $p$, we apply Lemma \ref{for any f} to $H = K^{d+1}(p)$ with the $(d+1)$st vertex part playing the special role. Under the contradiction assumption, no $m$-tuple of $F_{d+1}$ has nonempty intersection, which rules out case~(\ref{intersecting case}) of Lemma~\ref{for any f}. We therefore obtain case~(\ref{separation case}), and we obtain an induced subhypergraph such that every $m$-tuple of the $(d+1)$st part is separable with respect to this subhypergraph. 
This application can be repeated another $d$ times, relabeling the vertex parts at each step so that the $(d+1)$st part is the one being separated, with the case (\ref{intersecting case}) excluded at the $j$th step because no $m$-tuple of $F_j$ has nonempty intersection under the contradiction assumption. This eventually yields
an induced  subhypergraph $H' = K^{d+1}(2)$ with vertex parts $W_1, \dots, W_{d+1}$ and edge set $\tilde{E}$ such that 
the unique pair of vertices in each part $W_i$ is separable with respect to $H'$. In particular, writing $W_i = \{u_i, v_i\}$, there exists a halfspace $\gamma_i\in \mathcal{C}$ such that $f(\tilde{E}_{u_i})\subset \gamma_i$ and $f(\tilde{E}_{v_i})\subset \bar{\gamma}_i$. 

Now, for $e = \{a_1, \dots, a_{d+1}\} \in \tilde{E}$ we have
\[f(e) \in A_1\cap \cdots \cap A_{d+1},\]
where \[A_i = 
\begin{cases}
\gamma_i & \text{when } a_i = u_i\\
\bar{\gamma}_i & \text{when } a_i = v_i.
\end{cases}\] 
Since $H' = K^{d+1}(2)$, every choice of $a_i \in \{u_i, v_i\}$ yields an edge in $\tilde{E}$, so all $2^{d+1}$ regions of the Venn diagram of $\gamma_1, \dots, \gamma_{d+1}$ are witnessed nonempty by a point $f(e)$, meaning that $\gamma_1, \dots, \gamma_{d+1}$ have a complete Venn diagram. This contradicts the assumption that the dual VC-dimension of the system of halfspaces is at most $d$.  \qed

\section{Acknowledgements}

We are indebted to two anonymous referees who gave substantial feedback and comments that helped improve this manuscript.

The work was partially supported by ERC Advanced Grants ``GeoScape'' 882971 and ``ERMiD'' 101054936. A. H. was supported by the Institute for Basic Science (IBS-R029-C1), Z. P. was supported by the GA\v CR grant no. 25-16847S. 
Part of the work was done during the {\em Erd{\H o}s Center Focused week: Combinatorial geometry in Radon convexity spaces}, which took place in Budapest in December 2023.


\begin{thebibliography}{}

\bibitem{akmm} N.~Alon, G.~Kalai, J.~Matou{\v s}ek, and R.~Meshulam. Transversal numbers for hypergraphs arising in geometry.  {Adv. in Appl. Math.} 29 (2002), 79--101.

\bibitem{akpq} N.~Alon and D.~J.~Kleitman. Piercing convex sets and Hadwiger--Debrunner (p,q) problem. Adv.~Math. 96 (1992), 103--112.

\bibitem{assouad} P. Assouad. Densit{\'e} et dimension.  
Ann. Inst. Fourier 33 (1983), 233--282.

\bibitem{col-hell} I.~B{\'a}r{\'a}ny. A generalization of Carath{\'e}odory's theorem. {Discrete Math.} 40 (1982), 141--152.

\bibitem{bara-book}  I.~B{\'a}r{\'a}ny. {\em Combinatorial Convexity.} Amer.~Math.~Soc., University Lecture Series 77, 2021.

\bibitem{BK-bull}  I.~B{\'a}r{\'a}ny and G.~Kalai. Helly-type problems. Bull.~Amer.~Math.~Soc. 59 (2022), 471--502.

\bibitem{lattice} I.~B{\'a}r{\'a}ny and J.~Matou{\v s}ek. A fractional Helly theorem for convex lattice sets. {Adv. Math.} 174 (2003), 227--235.

\bibitem{chepoi24} V.~Chepoi. Separation axiom $S_3$ for geodesic convexity in graphs. arXiv:2405.07512


\bibitem{JPD} J.~P.~Doignon. 
Convexity in cristallographical lattices. {J. Geom.} 3 (1973), 71--85.

\bibitem{eckhoff} J.~Eckhoff. An upper-bound theorem for families of convex sets. Geom. Dedicata 19 (1985), 217--227.

\bibitem{eckhoff-box} J.~Eckhoff. The partition conjecture. Discrete Math. 221 (2000), 61--78.

\bibitem{ErdSim} P.~Erd{\H o}s and M.~Simonovits. Supersaturated graphs and hypergraphs. {Combinatorica} 3 (1983), 181--192.

\bibitem{gey} S.~Gey. Vapnik-Chervonenkis dimension of axis-parallel cuts. Communications in Statistics -- Theory and Methods, 47 (2018), 2291--2296.

\bibitem{helly} E.~Helly. {\"U}ber mengen konvexer k{\"o}rper mit gemeinschaftlichen punkte. {Jahresber. Deutsch. Math.-Verein.} 32 (1923), 175--176.

\bibitem{holmsen} A.~F.~Holmsen. Large cliques in hypergraphs with forbidden substructures. Combinatorica 40 (2020), 527--537.

\bibitem{holmsen-lee} A.~F.~Holmsen and D.~Lee. Radon numbers and the fractional Helly theorem. Israel~J.~Math. 241 (2021), 433--447.

\bibitem{jamison} R.E. Jamison. A general theory of convexity. Dissertation, University of Washington, Seattle, 1974.

\bibitem{kalai} G.~Kalai. Intersection patterns of convex sets. Israel~J.~Math. 48 (1984), 175--195.

\bibitem{katch-liu} M.~Katchalski and A.~Liu. A problem of geometry in $\mathbb{R}^n$. {Proc. Amer. Math. Soc.} 75 (1979), 284--288.

\bibitem{levi} F.~W.~Levi. On Helly's theorem and the axioms of convexity. {J. Indian Math. Soc.} 15 (1951), 65--76.

\bibitem{hemispaces} J.-E. Martinez-Legaz and I. Singer. \newblock The structure of hemispaces in $\mathbb R^n$. 
Linear Algebra Appl. 110 (1998), 117--179.

\bibitem{mato-gd} J.~Matou{\v s}ek. {\em Geometric Discrepancy}, Springer Algorithms and combinatorics 18, 1999.

\bibitem{mato} J.~Matou{\v s}ek. {\em Lectures on Discrete Geometry}, Springer GTM 212, 2002.

\bibitem{mato-vc} J.~Matou{\v s}ek. Bounded VC-dimension implies a fractional Helly theorem. {Discrete~Comput.~Geom.} 31 (2004), 251--255.


\bibitem{moran} S.~Moran and A.~Yehudayoff. On weak $\varepsilon$-nets and the Radon number. Discrete~Comput.~Geom. 64 (2020), 1125--1140.

\bibitem{onn} S. Onn. On the geometry and computational complexity of {R}adon partitions in the integer lattice. {SIAM J. Discrete Math.} 4 (1991), 436--446.

\bibitem{sierksma} G. Sierksma. Relationships between Carath{\'e}odory, Helly, Radon and exchange numbers of convexity spaces. Nieuw Arch. Wist. 3 (1977), 115--132.

\bibitem{vandevel} M.~L.~J.~van~de~Vel, {\em Theory of convex structures}, Vol. 50 of North-Holland Mathematical Library, North-Holland, 1993. 

\end{thebibliography}


\end{document}